\documentclass[10pt]{article}
\usepackage{amsmath,amssymb,amsthm,amsfonts,graphicx,mathrsfs,mathabx,color,enumitem,tikz,float,xfrac}
\usepackage[margin = 1.25in]{geometry}
\usetikzlibrary{patterns}

\newcommand{\N}{\mathbb{N}}

\renewcommand{\S}{{\mathscr S}}
\newcommand{\T}{{\mathscr T}}

\newcommand{\D}{{\cal D}}
\newcommand{\E}{{\cal E}}

\newtheorem{theorem}{Theorem}[section]
\newtheorem{lemma}[theorem]{Lemma}

\newtheorem{proposition}[theorem]{Proposition}
\newtheorem{definition}[theorem]{Definition}
\newtheorem{corollary}[theorem]{Corollary}
\newtheorem{question}[theorem]{Question}

\theoremstyle{remark}
\newtheorem{remark}[theorem]{\bf Remark}
\newtheorem{example}[theorem]{\bf Example}
\newtheorem*{ack}{Acknowledgements}

\DeclareMathOperator{\rk}{rk}
\DeclareMathOperator{\crk}{corank}

\DeclareMathOperator{\hocolim}{hocolim}
\newcommand{\lat}{L}

\DeclareFontFamily{OT1}{pzc}{}
\DeclareFontShape{OT1}{pzc}{m}{it}{<-> s * [1.150] pzcmi7t}{}
\DeclareMathAlphabet{\mathpzc}{OT1}{pzc}{m}{it}
\newcommand{\Top}{\mathpzc{Top}}

\begin{document}
\title{Face numbers of Engstr\"om representations of matroids}

\author{Steven Klee  \\ Seattle University \\ 
{\tt klees@seattleu.edu} \and Matthew T. Stamps \\ Royal Institute of Technology \\ {\tt stamps@math.kth.se}}

\date{\today}

\maketitle

\begin{abstract}  A classic problem in matroid theory is to find subspace arrangements, specifically hyperplane and pseudosphere arrangements, whose intersection posets are isomorphic to a prescribed geometric lattice.  Engstr\"{o}m recently showed how to construct an infinite family of such subspace arrangements, indexed by the set of finite regular CW complexes.  In this note, we compute the face numbers of these representations (in terms of the face numbers of the indexing complexes) and give upper bounds on the total number of faces in these objects.  In particular, we show that, for a fixed rank, the total number of faces in the Engstr\"om representation corresponding to a codimension one homotopy sphere arrangement is bounded above by a polynomial in the number of elements of the matroid with degree one less than its rank.  \end{abstract}

\section{Introduction}

The \emph{Topological Representation Theorem} for matroids, proved first by Folkman and Lawrence \cite{BLSWZ} for oriented matroids and then by Swartz \cite{Swartz03} for all matroids, asserts that every (oriented) matroid can be realized by a codimension one pseudo/homotopy sphere arrangement.  For a given matroid $M$, Anderson \cite{Anderson2010} and Engstr\"om \cite{Engstrom10} have given different constructions of an arrangement representing $M$.  Both approaches take as input the underlying geometric lattice of $M$.  The former utilizes flags along with tools presented in \cite{Bjorner95} whereas the latter utilizes homotopy colimits of diagrams of spaces \cite{WZZ99}.  

While there are advantages to each approach, one significant difference is that the construction in \cite{Anderson2010} can have arbitrarily high dimension for a fixed rank, while the construction in \cite{Engstrom10} keeps the dimension of the arrangement in direct correspondence with the rank function of the underlying matroid.  The idea behind each is to glue homotopy spheres together to form an arrangement by including either higher dimensional cells or mapping cylinders of an appropriate dimension.    It is then natural to ask, from a computational perspective, how expensive it is to keep the dimension down?  In other words, how many more faces are introduced by Engstr\"om's mapping cylinder approach for controling the dimension?  In this note, we show that the answer to this question is not many when the rank is fixed.  Specifically, we prove that the total number of faces in an Engstr\"om representation of matroid $M$ is a polynomial in the number of elements of $M$ with degree at most one less than its rank.  We also describe the asymptotic behavior of these numbers as the rank increases.

\section{Preliminaries}

We begin with the necessary definitions and theorems concerning geometric lattices and (oriented) matroids \cite{BLSWZ,Oxley,White}; homotopy colimits of diagrams of spaces \cite{WZZ99}; and Engstr\"{o}m representions of matroids \cite{Engstrom10,Stamps11}.  

\subsection{Geometric Lattices and (Oriented) Matroids}

Let $E$ be a finite set.  A \textit{sign vector} is a vector $X \in \{+,0,-\}^E$.  The \textit{zero set} of a sign vector is $z(X) := \{e \in E: X_e = 0\}$, and its support is $\underline{X} := \{e \in E: X_e \neq 0\}$.  The \textit{opposite} of a sign vector $X$ is the vector $-X$ whose entries are the opposites of those in $X$; that is,  
$$(-X)_e = \begin{cases} +, & \text{if }X_e = -, \\ -, & \text{if }X_e = +, \\ 0, & \text{if }X_e = 0.\end{cases}$$  The \textit{composition} of two sign vectors $X$ and $Y$ is the sign vector $X \circ Y$ defined by $$(X \circ Y)_e = \begin{cases} X_e, & \text{ if } X_e \neq 0, \\ Y_e, & \text{ otherwise.} \end{cases}$$  The \textit{separation set} of $X$ and $Y$ is $S(X,Y):=\{e \in E: X_e = -Y_e \neq 0\}$.  

An \textit{oriented matroid} $M$ consists of a finite set $E$, called the \emph{ground set}, and a collection of \textit{covectors} $\mathcal{L} \subseteq \{+,0,-\}^E$ satisfying
\begin{enumerate}[leftmargin = 50pt, rightmargin = 50pt]\itemsep -2pt
\item[(L0)] $0 \in \mathcal{L}$;
\item[(L1)] if $X \in \mathcal{L}$, then $-X \in \mathcal{L}$;
\item[(L2)] if $X,Y \in \mathcal{L}$, then $X \circ Y \in \mathcal{L}$; and
\item[(L3)] if $X,Y \in \mathcal{L}$ and $e \in S(X,Y)$, then there exists $Z \in \mathcal{L}$ such that $Z_e = 0$ and $Z_f = (X \circ Y)_f = (Y \circ X)_f$ for all $f \notin S(X,Y)$.
\end{enumerate}
If we ignore the sign data in $\mathcal{L}$, that is, if we consider only the set $L=\{z(X): X \in \mathcal{L}\}$, then we obtain the underlying matroid of $M$.  In this setting, (L1) becomes trivial and (L0), (L2), and (L3) become a set of axioms for matroids.  In general, a \emph{matroid} $M$ consists of a finite set $E$, and a collection of \emph{flats} $L \subseteq 2^E$, satisfying  \begin{enumerate}[leftmargin = 50pt, rightmargin = 50pt]\itemsep -2pt
\item[(F1)] $E \in L$; 
\item[(F2)] if $X,Y \in L$, then $X \cap Y \in L$; and 
\item[(F3)] for every $X \in L$, the set of all $Y \setminus X$ where $X \subsetneq Y \in L$ and there is no $Z \in L$ such that  $X \subsetneq Z \subsetneq Y$ forms a partition of $E \setminus X$.
\end{enumerate}
For readers who prefer to think of matroids in terms of independent sets, the flats of $M$ are the rank-maximal subsets of $E$, that is, $\rk(X \cup e) > \rk(X)$ for any $e \in E \setminus X$.  It is well known, see \cite{Bjorner82,Stanley}, that $L$ forms a graded \emph{geometric lattice}, meaning
\begin{enumerate}\itemsep -2pt
\item[(1)] $\lat$ is semimodular (that is, $\rk(p) + \rk(q) \geq \rk(p \wedge q) + \rk(p \vee q)$ for all $p,q \in L$) and
\item[(2)] every element  of $\lat$ is a join of atoms.
\end{enumerate}

If $M$ is oriented, the map $z: \mathcal{L} \rightarrow L$ is a cover-preserving, order-reversing surjection of $\mathcal{L}$ onto $L$, see \cite[Proposition 4.1.13]{BLSWZ}.   For every subset $X \subseteq E$, let $$\overline{X} = \bigcap_{X \subseteq Y \in L} Y$$ denote the \emph{closure} of $X$ in $M$ and define the \emph{rank} of $X$ to be the rank of $\overline{X}$ in $L$.  A \emph{weak map} between matroids $M$ and $N$ is a function $\tau : E(M) \to E(N)$ such that $\rk_M(X) \geq \rk_N(\tau(X))$ for all $X \subseteq E(M)$.  Every weak map $\tau : M \to N$ induces a weakly rank-decreasing, order-preserving map $\overline{\tau} : \lat(M) \to \lat(N)$ given by $X \mapsto \overline{\tau(X)}$ for all $X \in \lat(M)$ ( see \cite{White}). 

\begin{lemma}[\cite{Stamps11}, Lemma 3] \label{surj_poset}
If $\tau : M \to N$ is a surjective weak map, then $\overline{\tau} : \lat(M) \to \lat(N)$ is a surjective poset map.
\end{lemma}

\subsection{Topological Representations of (Oriented) Matroids}

One of the most natural families of oriented matroids arises from real hyperplane arrangements. Let $\mathcal{A} = \{H_1, H_2, \ldots, H_n\}$ be an essential arrangement of hyperplanes in $\mathbb{R}^{r}$, each of which contains the origin.  Intersecting $\bigcup_i H_i$ with the unit sphere in a generic hyperplane $\mathcal{H} \subseteq \mathbb{R}^r$ yields a cell decomposition of $S^{r-2}$.  Since each hyperplane $H_i$ has a positive side and a negative side, we may associate a sign vector in $\{+,-,0\}^n$ to each cell of this decomposition.  The collection of such sign vectors, together with the zero vector, satisfy the covector axioms of an oriented matroid \cite{BLSWZ}.  If an oriented matroid arises from a hyperplane arrangement in this way, we say that $M$ is \textit{realizable}.  

The signed covectors of an oriented matroid form a lattice whose componentwise order relations are induced by declaring that $0 < +,-$.  Folkman and Lawrence \cite{FL78} proved that this lattice is isomorphic to the face poset of a cell decomposition of $S^{r-2}$.  This cell decomposition is known as the \textit{Folkman-Lawrence representation of $M$}, and we denote it as $\S M$.  In the case that $M$ is realizable, this decomposition is the natural one formed by intersecting the sphere with the corresponding hyperplane arrangement.  

Without the orientation data, it is unknown how to find such a decomposition of $S^{r-2}$ for prescribed matroid $M$.  Instead, one can hope to construct a cell complex containing an arrangement of homotopy spheres whose intersection lattice matches $\lat(M)$.   Engstr\"om \cite{Engstrom10} gave one such construction using diagrams of spaces and homotopy colimits.  Diagrams of spaces provide a convenient way to arrange topological spaces according to some prescribed combinatorial information, and homotopy colimits are the natural tool for gluing those spaces together with respect to the given information.  

\begin{definition}
A \emph{$P$-diagram of spaces} $\D$ consists of the following data: \begin{itemize}\itemsep -2pt \item a finite poset $P$,  \item a CW complex $D(p)$ for every $p \in P$, \item a continuous map $d_{pq} : D(p) \to D(q)$ for every pair $p \geq q$ of $P$ satisfying \\ $d_{qr}\circ d_{pq}(x) = d_{pr}(x)$ for every triple $p \geq q \geq r$ of $P$ and $x \in D(p)$. \end{itemize}  
\end{definition}
To every diagram $\D$, we associate a topological space via a (homotopy) colimit.  In our setting, the \emph{homotopy colimit} of a diagram $\D : P \to \Top$ is the space $$\hocolim_P \D = \coprod\limits_{p \in P} (\Delta(P_{\leq p}) \times D(p)) \ /  \sim$$ where $\Delta(P_{\leq p})$ is the \emph{order complex} of $P_{\leq p}$ and $\sim$ is the transitive closure of the relation $(a,x) \sim (b,y)$ for each $a \in \Delta(P_{\leq p})$, $b \in \Delta(P_{\leq q})$, $x \in D(p)$ and $y \in D(q)$ if and only if $p \geq q$, $d_{pq}(x) = y$, and $a = b$. 

\begin{example}\label{ex:hocolim}
Let $P$ be the dual of the face poset (the face poset ordered by reverse inclusion) of the simplicial complex shown in Figure \ref{figure:complex} (left).

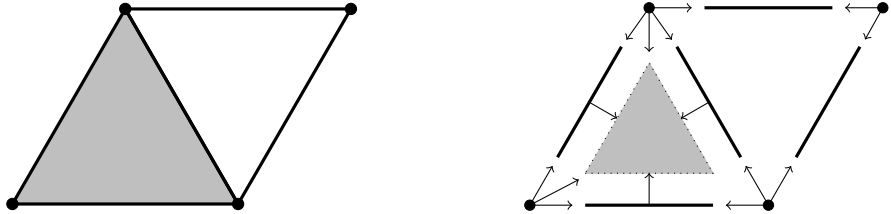
\begin{figure}[h]
\begin{center}
\begin{minipage}{.45\textwidth}
\begin{center}
\begin{tikzpicture}[scale=1.5]
\filldraw[color=gray!50] (0,0) -- (2,0) -- (1, 1.73);
\draw[very thick](2,0) -- (3,1.73) -- (1,1.73) -- (2,0);
\draw[very thick] (0,0) -- (2,0) -- (1,1.73) -- (0,0);
\filldraw[black] (0,0) circle (.05);
\filldraw[black] (2,0) circle (.05);
\filldraw[black] (1,1.73) circle (.05);
\filldraw[black] (3,1.73) circle (.05);
\end{tikzpicture}
\end{center}
\end{minipage}
\begin{minipage}{0.45\textwidth}
\begin{center}
\begin{tikzpicture}[scale=0.85]
\filldraw[color=gray!50] (0,0) -- (2,0) -- (1,1.73);
\draw[dotted] (0,0) -- (2,0) -- (1,1.73) -- (0,0);
\filldraw[black] (-.866, -.5) circle (.08);
\filldraw[black] (1, 2.59) circle (.08);
\filldraw[black] (2.866, -.5) circle (.08);
\filldraw[black] (4.655, 2.59) circle (.08);
\draw[very thick] (-.433,.25) -- (1-.433,1.73+.25);
\draw[very thick] (2+.433,.25) -- (1+.433, 1.73+.25);
\draw[very thick] (0,-.5) -- (2,-.5);
\draw[very thick] (2.866+.433, .25) -- (4.73-.433, 1.73+.25);
\draw[very thick] (1.866, 2.59) -- (3.866, 2.59);
\draw[->] (-.866, -.5) -- (-.1,-.1);
\draw[->] (-.866,-.5) -- (-.2, -.5);
\draw[->] (-.866,-.5) -- (-.512, .112);
\draw[->] (.066, 1.11) -- (.066+.433, 1.11-.25);
\draw[->] (1,2.59) -- (.646, 2.09);
\draw[->] (1,2.59) -- (1.354, 2.09);
\draw[->] (1,2.59) -- (1,1.9);
\draw[->] (1,2.59) -- (1.666,2.59);
\draw[->] (1.933, 1.11) -- (1.933-.433, 1.11-.25);
\draw[->] (1,-.5) -- (1,0);
\draw[->] (2.866, -.5) -- (2.2,-.5);
\draw[->] (2.866, -.5) -- (2.512, .112);
\draw[->] (2.866, -.5) -- (3.22, .112);
\draw[->] (4.655, 2.59) -- (4.065, 2.59);
\draw[->] (4.655, 2.59) -- (4.378, 2.094);
\end{tikzpicture}
\end{center}
\end{minipage}
\caption{A simplicial complex (left) viewed as a diagram of spaces (right).}
\vspace{-0.25cm}
\label{figure:complex}
\end{center}
\end{figure}

One can form a diagram of spaces $\D$ over $P$ by declaring that $D(p) = p$ for each face $p \in \Delta$ and that $d_{pq}$ is the natural inclusion $p \hookrightarrow q$ for each $p \subseteq q$ in $\Delta$.  The resulting diagram is shown in Figure 1 (right).  Its homotopy colimit is illustrated in Figure \ref{figure:homotopy-colimit} (left), along with a decomposition into open stars (right).  This decomposition will be useful for the proof of Theorem \ref{thm:main}.

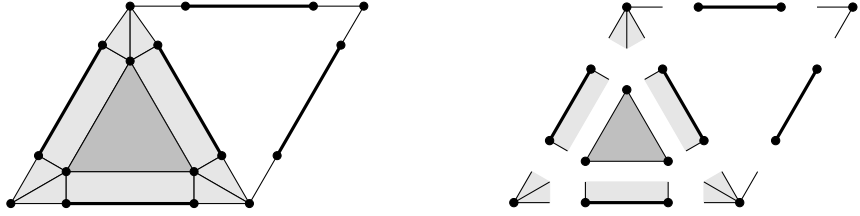
\begin{figure}[h]
\begin{center}
\begin{minipage}{0.45\textwidth}
\begin{center}
\begin{tikzpicture}[scale=0.85]
\filldraw[fill=gray!20] (0,0) -- (1,1.73) -- (1-.433, 1.73+.25) -- (-.433, .25) -- (0,0);
\filldraw[fill=gray!20] (0,0) -- (-.866,-.5) -- (-.433,.25) -- (0,0);
\filldraw[fill=gray!20] (0,0) -- (-.866,-.5) -- (0,-.5) -- (0,0);
\filldraw[fill=gray!20] (1,2.59) --  (1-.433,1.73+.25) -- (1,1.73) -- (1,2.59);
\filldraw[fill=gray!20] (1,2.59) --  (1+.433, 1.73+.25) -- (1,1.73) -- (1,2.59);
\filldraw[fill=gray!20]  (1,1.73) -- (1+.433, 1.73+.25) --  (2+.433,.25) -- (2,0) -- (1,1.73);
\filldraw[fill=gray!20] (0,0) -- (2,0) -- (2,-.5) -- (0,-.5) -- (0,0);
\filldraw[fill=gray!20] (2.866, -.5) -- (2,0) -- (2+.433,.25) -- (2.866, -.5);
\filldraw[fill=gray!20] (2.866, -.5) -- (2,0) -- (2,-.5) -- (2.866, -.5);
\draw (1,2.59) -- (2,2.59);
\draw (3.866,2.59) -- (4.655, 2.59);
\draw (4.73-.433, 1.73+.25) -- (4.655, 2.59);
\draw (2.866, -.5) -- (2.866+.433, .25);
\filldraw[color=gray!50] (0,0) -- (2,0) -- (1,1.73);
\draw (0,0) -- (2,0) -- (1,1.73) -- (0,0);
\filldraw[black] (0,0) circle (.0625);
\filldraw[black] (2,0) circle (.0625);
\filldraw[black] (1,1.73) circle (.0625);
\filldraw[black] (-.866, -.5) circle (.0625);
\filldraw[black] (1, 2.59) circle (.0625);
\filldraw[black] (2.866, -.5) circle (.0625);
\filldraw[black] (4.655, 2.59) circle (.0625);
\draw[very thick] (-.433,.25) -- (1-.433,1.73+.25);
\filldraw[black] (-.433,.25) circle (.0625);
\filldraw[black] (1-.433, 1.73+.25) circle (.0625);
\draw[very thick] (2+.433,.25) -- (1+.433, 1.73+.25);
\filldraw[black] (2+.433, .25) circle (.0625);
\filldraw[black] (1+.433, 1.73+.25) circle (.0625);
\draw[very thick] (0,-.5) -- (2,-.5);
\filldraw[black] (0,-.5) circle (.0625);
\filldraw[black] (2,-.5) circle (.0625);
\draw[very thick] (2.866+.433, .25) -- (4.73-.433, 1.73+.25);
\filldraw[black] (2.866+.433, .25) circle (.0625);
\filldraw[black] (4.73-.433, 1.73+.25) circle (.0625);
\draw[very thick] (1.866, 2.59) -- (3.866, 2.59);
\filldraw[black] (1.866,2.59) circle (.0625);
\filldraw[black] (3.866,2.59) circle (.0625);
\end{tikzpicture}
\end{center}
\end{minipage}
\begin{minipage}{0.4\textwidth}
\begin{center}
\begin{tikzpicture}[scale=0.55]
\filldraw[fill=gray!50] (0,0) -- (2,0) -- (1,1.73) -- (0,0);
\filldraw[black] (0,0) circle (.1);
\filldraw[black] (2,0) circle (.1);
\filldraw[black] (1,1.73) circle (.1);
\filldraw[fill=gray!20, color=gray!20] (-.433,.25) -- (1-.433,1.73+.25) --  (.134, 1.73+.5) --(-.866,.5) -- (-.433, .25);
\draw[very thick] (-.866,.5) -- (.135,1.73+.5);
\draw (-.866,.5) -- (-.433,.25);
\draw (.135,1.73+.5) -- (.135+.433, 1.73+.25);
\filldraw[fill=gray!20, color=gray!20] (2+.433,.25) -- (1+.433, 1.73+.25) -- (1+.866, 1.73+.5) -- (2+.866,.5) -- (2+.433,.25);
\draw[very thick] (1+.866, 1.73+.5) -- (2+.866, .5);
\draw  (1+.433, 1.73+.25) -- (1+.866, 1.73+.5);
\draw (2+.866,.5) -- (2+.433,.25);
\draw[fill=gray!20, color=gray!20] (0,-.5) -- (2,-.5) -- (2,-1) -- (0,-1) -- (0,-.5);
\draw[very thick] (0,-1) -- (2,-1);
\draw (0,-1) -- (0,-.5); 
\draw (2,-1) -- (2,-.5);
\draw[very thick] (4.598,.5) -- (5.598,1.73+.5);
\filldraw[black] (4.598,.5) circle (.1);
\filldraw[black] (5.598,1.73+.5) circle (.1);
\draw[very thick] (1+1.73, 2+1.73) -- (3+1.73, 2+1.73);
\filldraw[black] (1+1.73, 2+1.73) circle (.1);
\filldraw[black] (3+1.73, 2+1.73) circle (.1);
\filldraw[gray!20] (-1.73,-1) -- (-1.73+.866,-1) -- (-1.73+.866,-.5) -- (-1.73,-1);
\filldraw[gray!20] (-1.73,-1) -- (-1.73+.433,-.25) -- (-1.73+.866,-.5);
\draw (-1.73,-1) -- (-1.73+.866,-1); 
\draw (-1.73,-1) -- (-1.73+.866,-.5);
\draw (-1.73,-1) -- (-1.73+.433,-.25);
\draw[fill=gray!20, color=gray!20] (1,2+1.73) -- (1,1+1.73) -- (1-.433,1.25+1.73);
\draw[fill=gray!20, color=gray!20] (1,2+1.73) -- (1,1+1.73) -- (1+.433,1.25+1.73);
\draw (1,2+1.73) -- (1,1+1.73);
\draw (1,2+1.73) -- (1.866,2+1.73);
\draw (1,2+1.73) -- (1-.433, 1.25+1.73);
\draw (1,2+1.73) -- (1+.433, 1.25+1.73);
\draw[fill=gray!20, color=gray!20] (2+1.73,-1) -- (2+.866,-.5) -- (2+.866,-1);
\draw (2+1.73,-1) -- (2+.866,-1);
\draw[fill=gray!20, color=gray!20] (2+1.73,-1) -- (2+.866,-.5) -- (2+.866+.433, -.25);
\draw (2+1.73,-1) -- (2+.866,-.5);
\draw (2+1.73,-1) -- (2+.866+.433, -.25);
\draw (2+1.73,-1) -- (2+1.73+.433, -.25);
\draw (6.464, 2+1.73) -- (6.464-.866, 2+1.73);
\draw (6.464, 2+1.73) -- (6.464-.433, 2+1.73-.75);
\filldraw[black] (-1.73,-1) circle (.1);
\filldraw[black] (1, 2+1.73) circle (.1);
\filldraw[black] (2+1.73, -1) circle (.1);
\filldraw[black] (6.464, 2+1.73) circle (.1);
\filldraw[black] (-.866,.5) circle (.1);
\filldraw[black] (1-.866,.5+1.73) circle (.1);
\filldraw[black] (0,-1) circle (.1);
\filldraw[black] (2,-1) circle (.1);
\filldraw[black] (1+.866, 1.73+.5) circle (.1);
\filldraw[black] (2+.866, .5) circle (.1);
\end{tikzpicture}
\end{center}
\end{minipage}
\caption{The homotopy colimit of $\Delta$ (left) decomposed as a union of open stars (right).}
\vspace{-0.25cm}
\label{figure:homotopy-colimit}
\end{center}
\end{figure}

\end{example}

The natural notion of a structure-preserving map between diagrams of spaces $\D : P \to \Top$ and $\E : Q \to \Top$ is a \emph{morphism of diagrams} $(f,\alpha) : \D \to \E$ consisting of a poset map $f : P \to Q$ together with a natural transformation $\alpha$ from $\D$ to $\E \circ f$.   Morphisms of diagrams induce continuous maps between the corresponding homotopy colimits (see \cite{WZZ99}).  In fact, these maps are completely explicit. If we write each point in $\Delta(P_{\leq p}) \times D(p)$ as $(\lambda_1 p_1 + \cdots + \lambda_k p_k , x)$ where $p_1 \leq p_2 \leq \cdots \leq p_k = p$, $\lambda_i \geq 0$, $\sum_i \lambda_i = 1$, and $x \in D(p)$, then $$f^\alpha(\lambda_1 p_1 + \cdots + \lambda_k p_k , x) = (\lambda_1 f(p_1) + \cdots + \lambda_k f(p_k) , \alpha_p(x))$$  is the desired map.  

\begin{remark}\label{cellular} If $\D$ and $\E$ are diagrams of CW complexes with cellular maps and the maps in $\alpha$ are all cellular, then $f^{\alpha}$ is cellular as well. \end{remark}

Let $M$ be a rank $r$ matroid and $\ell$ be a rank- and order-reversing poset map from $\lat(M)$ to $B_r$, the boolean lattice on $[r]$.  For every locally finite, regular $CW$ complex $X$ and $\sigma \in B_r$, define a $B_r$-diagram $\D_X$ by $$D_X(\sigma) = \Asterisk_{i=1}^r \begin{cases} X & \text{if $i \in \sigma$} \\ \emptyset & \text{if $i \notin \sigma$} \end{cases}$$ with the corresponding inclusions.  This gives an $\lat(M)$-diagram $\D_X(M,\ell) := \D_X \circ \ell$.  The \emph{Engstr\"om representation} of a pair $(M,\ell)$, indexed by $X$, is the space $$ \T_X(M,\ell) := \hocolim_{\lat(M)} \D_X(M,\ell).$$

Engstr\"om \cite[Theorem 3.7]{Engstrom10} showed that the homotopy type of $\T_X(M,\ell)$ is independent of $\ell$, so it will be convenient to fix a canonical choice $\hat{\ell} : \lat(M) \to B_r$ given by $\hat{\ell}(p) = \{1,2,\dots,\crk(p)\}$ and abbreviate $\T_X(M,\hat{\ell})$ to $\T_X M$.  

\begin{theorem}[\cite{Stamps11}, Corollary 3]
For every weak map $\tau : M \to N$, there exists a natural transformation $\iota : \D_X(M) \to \D_X(\tau(M))$ such that $\overline{\tau}^{\iota} : \T_X M \to \T_X N$ is a continuous map.
\end{theorem}

The following result is an immediate corollary of Lemma \ref{surj_poset} and Remark \ref{cellular}.

\begin{corollary}\label{surj_cell}
If $\tau : M \to N$ is surjective, then $\overline{\tau}^{\iota} : \T_X M \to \T_X N$ is a surjective cellular map.
\end{corollary}

\section{Results}

This is the main section of the paper where we compute the $f$-polynomial of the Engstr\"om representation $\T_X M$ in terms of the $f$-polynomial of its indexing complex $X$.  We also show that the total number of faces in $\T_{S^0} M$ is bounded above by a  polynomial in the number of elements of $M$ with degree at most one less than its rank and compare the behavior of this number to that of the Folkman-Lawrence representation $\S M$ in the special case that $M$ is uniform.  

\subsection{Face Polynomials of Engstr\"om Representations}
 
The \textit{$f$-polynomial} of a finite CW complex $\Delta$ is the polynomial $f(\Delta; t):=\sum_{i \geq 0}f_{i-1}(\Delta)t^i$, where $f_{i-1}(\Delta)$ denotes the number of $(i-1)$-dimensional faces in $\Delta$.  By convention, we set $f_{-1}(\Delta)=1$, corresponding to the empty face, for any nonempty complex $\Delta$, and make use of the following three standard formulas for computing $f$-polynomials.  If $\Delta$ and $\Gamma$ are regular CW complexes, then \begin{align} 
f(\Delta * \Gamma; t) &= f(\Delta; t) \cdot f(\Gamma; t), \\ 
f(\Delta \times \Gamma; t) &= \frac{(f(\Delta; t)-1) \cdot (f(\Gamma; t)-1)}{t} + 1, \\ 
f(\Delta \sqcup \Gamma; t) &= f(\Delta; t) + f(\Gamma; t) - 1. 
\end{align} 

\begin{theorem}\label{thm:main}
For any given CW complex $X$ and matroid $M$, $$f(\T_X M; t) = 1 + \sum\limits_{p \in \lat(M)} \frac{(f(\Delta^{\circ}_M(p); t)-1) \cdot (f(X; t)^{\crk(p)} - 1)}{t},$$ where $\Delta^{\circ}_M(p)$ is the open star of $p$ in $\Delta(\lat(M)_{\leq p})$.
\end{theorem}

\begin{proof}  
The Engstr\"om representation $\T_X M$ arises naturally as a quotient of the space $Y$ that is a disjoint union of spaces indexed by the elements $p \in \lat(M)$.  Each such $p$ contributes a component to $Y$ that is the product of $\Delta(\lat(M)_{\leq p})$ with a $\crk(p)$-fold join of the space $X$.  The $f$-polynomial of this space is easily computed using formulas (1)-(3) as $$f(Y; t) =  \sum\limits_{p \in \lat(M)} \left(\frac{(f(\Delta(\lat(M)_{\leq p}); t)-1) \cdot (f(X; t)^{\crk(p)} - 1)}{t} + 1\right) - (|\lat(M)| - 1).$$ Thus, the only difficulty in computing the $f$-polynomial of $\T_X M$ is accounting for the quotient $\sim$.  One could proceed na\"ively by sieving out the over-counted cells identified by $\sim$, but it is simpler to observe that $\T_X M$ can be decomposed nicely into a disjoint union of half-open spaces, as illustrated in Figure \ref{figure:homotopy-colimit} of Example \ref{ex:hocolim}.  In particular, since every vertex $q \in \Delta(\lat(M)_{\leq p}) \setminus p$ is identified with itself in $\Delta(\lat(M)_{\leq q})$, it suffices to consider only the space $\Delta^{\circ}_M(p)$ that is the union of all open cells in $\Delta(\lat(M)_{\leq p})$ whose closures contain $p$.  The desired result follows by replacing $\Delta(\lat(M)_{\leq p})$ with $\Delta^{\circ}_M(p)$ for each $p \in \lat(M)$.  
\end{proof}

At first glance, the formula in Theorem \ref{thm:main} may appear rather opaque, but it is quite easy to use for a number of important classes of matroids, as we hope to illustrate in Examples \ref{ex:main} and \ref{ex:fano}.  Also, note that for $\hat{1} = E(M) \in \lat(M)$, $D_X(\hat{1}) = \emptyset$, which makes the summand in Theorem \ref{thm:main} equal to zero and hence, it suffices to take the sum over $\lat(M) \setminus \hat{1}$.

\begin{example}\label{ex:main}  Let $U_{r,n}$ denote the rank $r$ \emph{uniform matroid} on $n$ elements, that is, the matroid on $[n]$ whose flats consist of $[n]$ and all subsets of size less than $r$. In this case, $L(U_{r,n}) \setminus \hat{1}$ is isomorphic to the subposet of $B_n$ consisting of all elements with rank less than $r$.  Thus for any $p \in L(U_{r,n}) \setminus \hat{1}$ with $\rk(p) = i$, $$\Delta^{\circ}_{L(U_{r,n})}(p) \times D_X(p) = \Delta^{\circ}_{B_i}(\hat{1}) \times X^{*(r-i)}.$$  As such, we note that $f(\Delta^{\circ}_{B_i}(\hat{1}); t) = 1 + \sum\limits_{k = 0}^{i} k! S(i+1,k+1) t^{k+1}$ since the number of $k$ dimensional faces of $\Delta^{\circ}_{B_i}(\hat{1})$ is exactly the number of chains of length $k$ in $B_i$ that contain $\hat{1}$.  This is exactly the number of ordered partitions of $\{\emptyset,1,2,\ldots,i\}$ with $k+1$ parts that have $\emptyset$ in the first part.  So, bringing everything together, Theorem \ref{thm:main} asserts that $$f(\T_X U_{r,n}; t) = 1 + \sum\limits_{i = 0}^{r} \binom{n}{i} \cdot \mathcal{F}_i(t) \cdot (f(X; t)^{r-i} - 1),$$ where $\displaystyle \mathcal{F}_i(t) = \frac{f(\Delta^{\circ}_{B_i}(\hat{1}); t) - 1}{t} = \sum\limits_{k=0}^{i} k! S(i+1,k+1)t^k$. 
\end{example}

\begin{remark}\label{rem:main} Since every rank $r$ matroid on $n$ elements is a surjective weak map image of $U_{r,n}$ \cite{White}, Corollary \ref{surj_cell} implies that the formula in Example \ref{ex:main} gives upper bounds for the $f$-polynomials of the Engstr\"om representations of any matroid.
\end{remark}

\subsection{The Total Number of Faces of an Engstr\"om Representation}

From here on, we restrict our attention to the Engstr\"om representations where $X = S^0$.  These are codimension one homotopy sphere arrangements and therefore the most natural objects to compare with the Folkman-Lawrence representations of oriented matroids.  

For a finite CW complex $\Delta$, let $f_{total}(\Delta) := f(\Delta; 1)$ denote the total number of faces in $\Delta$. 

\begin{corollary}\label{cor:total}
If $M$ is a rank $r$ matroid on $n$ elements, then $$f_{total}(\T_{S^0} M) \leq f_{total}(\T_{S^0} U_{r,n}) = 1+ \sum\limits_{i = 0}^{r} \binom{n}{i} \cdot (2 F_i - 0^i) \cdot (3^{r-i}-1),$$ where $F_i = \sum\limits_{k = 0}^{i} k!S(i,k)$ is the $i$-th ordered Bell (or Fubini) number.  
\end{corollary}

\begin{proof}
The inequality on the left follows immediately from Remark \ref{rem:main} and the equality on the right follows from Example \ref{ex:main} by setting $f_{S^0}(t) = 1 + 2t$.  We then evaluate at $t = 1$, and observe that \begin{align*} \sum\limits_{k = 0}^i k! S(i+1,k+1) &= \sum\limits_{k = 0}^i \left((k+1)! S(i,k+1) + k! S(i,k) \right) \\ &= S(i,i+1) + 2 \cdot \sum\limits_{k = 0}^i k! S(i,k) - S(i,0), \end{align*} $S(i,i+1) = 0$, and $S(i,0) = 0^i$.
\end{proof}

\begin{lemma}\label{lem:Engstrom}
For every $r \in \N$, $f_{total}(\T_{S^0} U_{r,n})$ is a polynomial of degree $r-1$ in $n$ whose leading coefficient is $\frac{4 \cdot F_{r-1}}{(r-1)!}$.
\end{lemma}

\begin{proof}
By Corollary \ref{cor:total}, $f_{total}(\T_{S^0} U_{r,n})$ is of the form $\sum\limits_{i=0}^r a_i \binom{n}{i}$ where $a_i = (2F_i-0^i)(3^{r-i}-1)$.  Since $\binom{n}{i}$ is a polynomial of degree $i$ in $n$ with leading term $\frac{1}{i!}$ and $a_r = 0$, it follows that $f_{total}(\T_{S^0} U_{r,n})$ is a degree $r-1$ polynomial in $n$ with leading term $\frac{a_{r-1}}{(r-1)!} = \frac{4F_{r-1}}{(r-1)!}$.         
\end{proof}

It follows from Corollary \ref{cor:total}  that the total number of faces of $\T_{S^0} M$ is bounded above by a  polynomial of degree $r-1$ in $n$.  In the case that $M$ is oriented, the total number of faces of $\S M$ is also bounded by a  polynomial of degree $r-1$ in $n$, as indicated in the following two results.  Therefore, the bound on the total number of faces in $\T_{S^0}(M)$ is the strongest upper bound one could hope to achieve.

\begin{proposition}[\cite{BLSWZ}, Chapter 4]\label{prop:FL}
For every oriented matroid $M$, $$f(\S M; t) = \sum\limits_{\substack{p,q \in L(M) \setminus \hat{0} \\ p \leq q}} | \mu_{L(M)}(p,q)| t^{\crk(p)}.$$
\end{proposition}

\begin{corollary}\label{cor:FL}
If $M$ is an oriented matroid of rank $r$ on $n$ elements, then $$f_{total}(\S M) \leq f_{total}(\S U_{r,n}) = 1+ 2 \sum\limits_{i = 0}^{r-1} \sum\limits_{k=0}^{r-i-1}  \binom{n}{i}  \binom{n-i-1}{k}.$$  
\end{corollary}

\begin{proof}
The inequality on the left follows from Remark \ref{rem:main}.  For the equality on the right, observe that for each element $p \in L(U_{r,n})$ with corank $i < r$, $L(U_{r,n})_{\geq p}$ is a truncated Boolean lattice on $n-i$ elements.  Therefore, \begin{align*}\sum\limits_{\substack{p,q \in L(U_{r,n}) \setminus \hat{0}\\ p \leq q}}  | \mu_{L(M)}(p,q)| &= | \mu_{L(M)}(p,\hat{1})| + \sum\limits_{\substack{p,q \in L(U_{r,n}) \setminus \{\hat{0},\hat{1}\} \\ p \leq q}}  | \mu_{L(M)}(p,q)| \\ &= \sum\limits_{k=0}^{r-i-1} (-1)^{k+r-i-1} \binom{n-i}{k} + \sum\limits_{k=0}^{r-i-1} \binom{n-i}{k} \\ &= \sum\limits_{k=0}^{r-i-1} 2 \binom{n-i-1}{k}. \end{align*} The result follows from Proposition \ref{prop:FL} by summing over all $p \in L(U_{r,n})$.
\end{proof}


\begin{lemma}\label{lem:FL}
For every $r \in \N$, $f_{total}(\S U_{r,n})$ is a  polynomial of degree $r-1$ in $n$ with leading coefficient $\frac{2^r}{(r-1)!}$.
\end{lemma}

\begin{proof}
We know $\binom{n}{i}$ is a polynomial of degree $i$ in $n$ with leading coefficient $\frac{1}{i!}$ and $\binom{n-i-1}{k}$ is a polynomial of degree $k$ in $n$ with leading coefficient $\frac{1}{k!}$.  So by Corollary \ref{cor:FL}, $f_{total}(\S U_{r,n})$ is a sum of polynomials of degree $i+k$ in $n$.  Since $k \leq r-i-1$, $f_{total}(\S U_{r,n})$ has degree at most $r-1$ and its $(r-1)$-st coefficient $a_{r-1}$ is determined by pairs $(i,k)$ with $k = r-i-1$.  Thus, $$a_{r-1} = 2 \sum\limits_{i=0}^{r-1}\frac{1}{i!(r-i-1)!} = \frac{2}{(r-1)!} \sum\limits_{i=0}^{r-1} \binom{r-1}{i} = \frac{2^r}{(r-1)!}. \qedhere$$  \end{proof}

For any oriented matroid $M$, let $\rho(M) := \frac{f_{total}(\T_{S^0} M)}{f_{total}(\S M)}$ denote the ratio of the total number of faces of the Engstr\"om representation of $M$ (with $X = S^0$) to that of its Folkman-Lawrence representation.      
By Lemmas \ref{lem:Engstrom} and \ref{lem:FL}, $f_{total}(\T_{S^0} U_{r,n})$ and $f_{total}(\S U_{r,n})$ are both degree $r-1$ polynomials in $n$ with leading coefficients $\frac{4 \cdot F_{r-1}}{(r-1)!}$ and $\frac{2^r}{(r-1)!}$, respectively.  This proves the following theorem. 

\begin{theorem}
For every $r \in \N$, $\displaystyle \lim_{n \to \infty} \rho(U_{r,n}) = \frac{F_{r-1}}{2^{r-2}}$ where $F_i$ is the $i$-th ordered Bell number. \qed
\end{theorem}

Barth\'el\'emy \cite{bart80} showed that $F_i \approx \frac{i!}{2(\ln 2)^{i+1}}$.  So, for large $n$, the ratio $\rho(U_{r,n}) \approx \frac{(r-1)!}{\ln 2 (\ln 4)^{r-1}}$ grows slower than the corresponding factorials by a small exponential factor.  This prompts us to ask the following question.

\begin{question}
Are there topological representations of matroids whose total face numbers grow more closely to those of the Folkman-Lawrence representations as the rank gets large?
\end{question}

One can replace the function $3^{r-i}-1$ in Corollary \ref{cor:total} with $2(r-i)$ by swapping $D_{S^0}(p)$ with the standard cell structure on $S^{r-i-1}$ consisting of two cells in each dimension, but this does not yield a meaningful improvement for large values of $n$ since the leading coefficient of the total number of faces of such a construction is still determined by the $(r-1)$-st term in the sum.

\begin{question}
How large can the ratio $\rho(M)$ be for an arbitrary rank $r$ matroid $M$ on $n$ elements?
\end{question}

One might hope to bound $\rho(M)$ in one direction by $\rho(U_{r,n})$, but it is easy to find small counterexamples preventing this.  It is not clear to the authors how $\rho(M)$ behaves with respect to $\rho(U_{r,n})$ as $n$ gets large.

\section{Appendix}

We conclude with an application of Theorem \ref{thm:main} to the Fano matroid.  

\begin{example}\label{ex:fano}
Recall that the Fano plane, $F$, is a rank three matroid on seven elements, as depicted in Figure \ref{fig:fano1},  and that $\lat(F)_{\leq p} \cong \lat(F)_{\leq q}$ for every $p,q \in \lat(F)$ with $\rk(p) = \rk(q)$.
\begin{figure}[H]
\begin{center} \scalebox{0.6}{\includegraphics{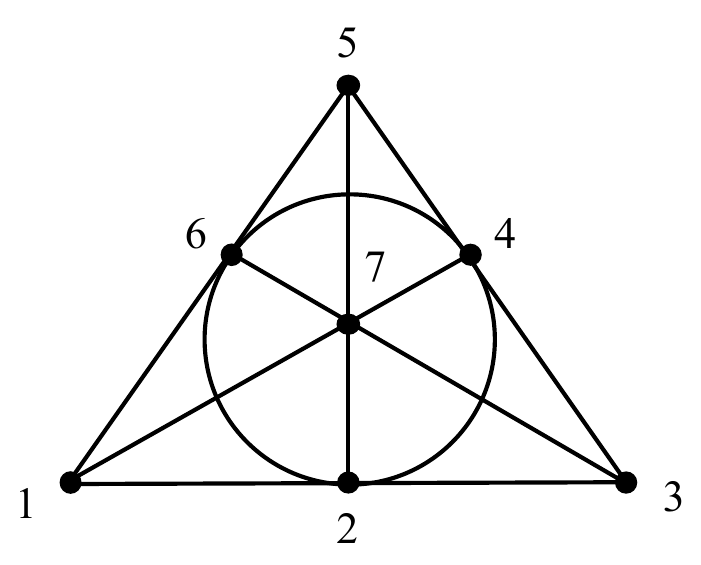} \hspace{1cm} \includegraphics{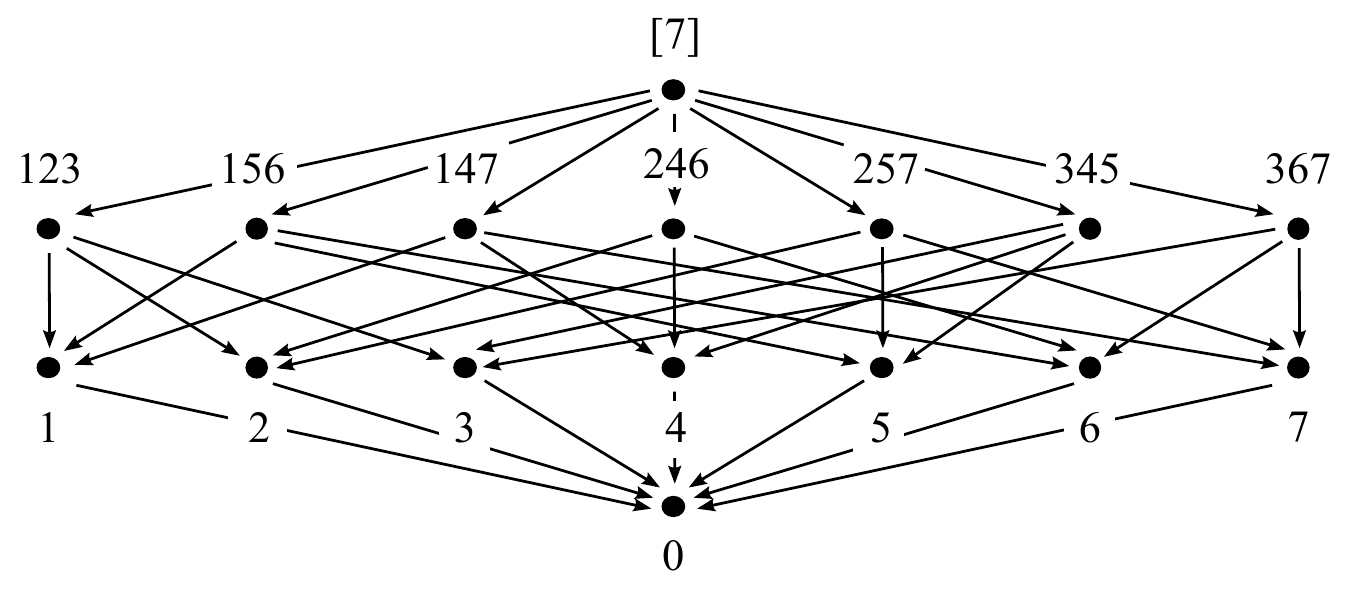}}\end{center}
\caption{The Fano plane (left) and its lattice of flats (right).} \vspace{-0.25cm} \label{fig:fano1}
\end{figure}  

\noindent We leave it to the reader to verify that $$f_{\Delta^{\circ}_F(p)}(t) = \begin{cases} 1 + t & \rk(p) = 0, \\ 1 + t + t^2 & \rk(p) = 1, \\ 1 + t + 4t^2 + 3t^3 & \rk(p) = 2. \end{cases}$$  Plugging this into Theorem \ref{thm:main}, along with $f_{S^0}(t) = 1 + 2t$, we get that \begin{align*} f_{\T_{S^0} F}(t) &= 1 + 1(6t+12t^2+8t^3) + 7(4t +8t^2+4t^3)+7(2t+8t^2+6t^3)+1(0) \\ &= 1+48t+124t^2+78t^3. \end{align*}  A portion of $\T_{S^0} F$, namely $\hocolim_{\lat(F)_{> \emptyset}} \D_{S^0}$, is illustrated in Figure \ref{fig:fano2}. \\

\begin{figure}[H]
\begin{center} \scalebox{0.35}{\includegraphics{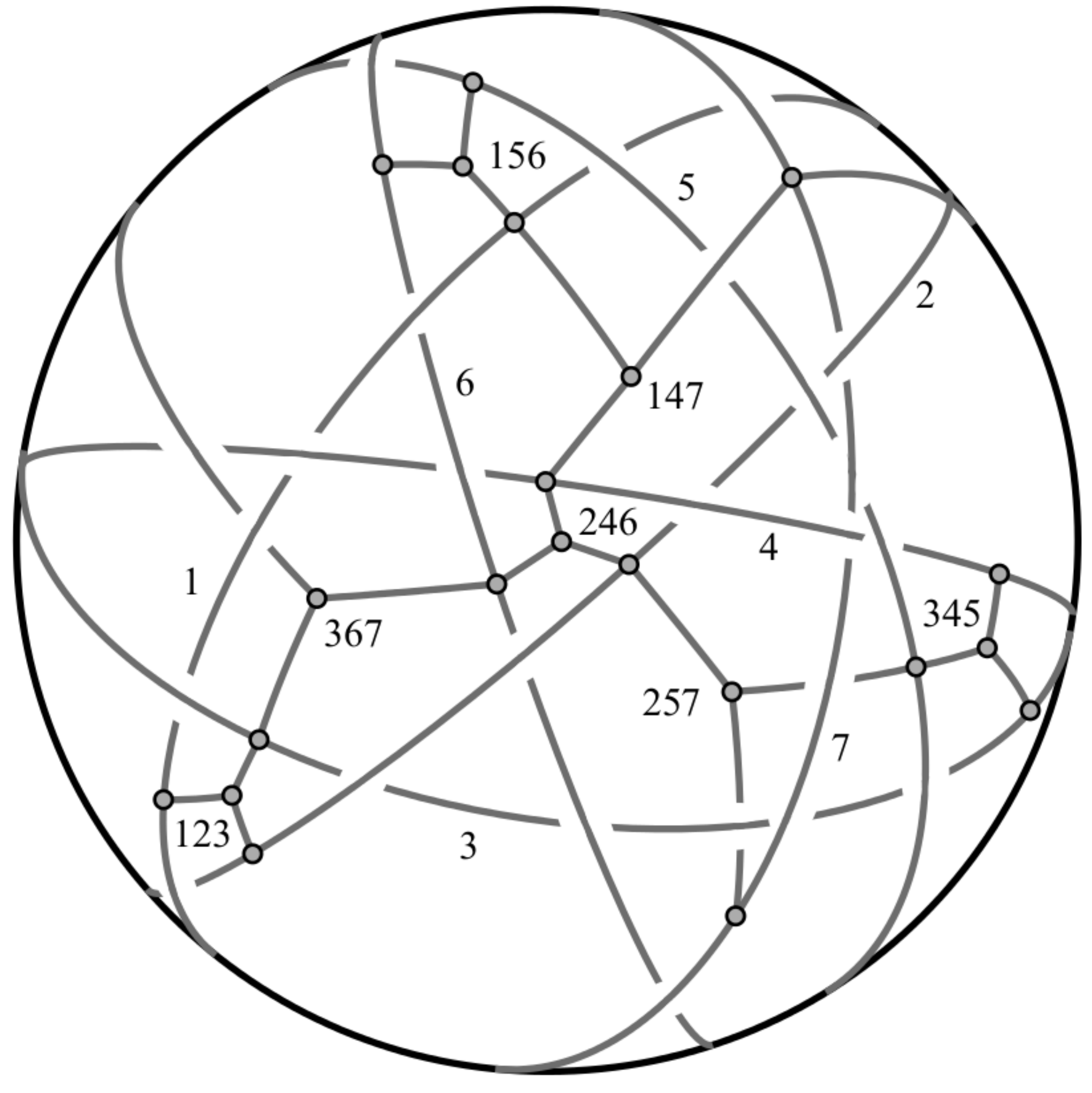}}\end{center}
\caption{A portion of the Engstr\"om representation of $F$ with $X = S^0$.} \vspace{-0.25cm} \label{fig:fano2}
\end{figure}
\end{example}

\begin{ack}
The authors wish to thank Louis Billera and Jes\'us De Loera for (independently) suggesting this project and for their several helpful conversations.  This research was partially supported by NSF Grant \#1159206.
\end{ack}

\bibliographystyle{amsplain}
\bibliography{references}{}

\end{document}